\documentclass[11pt]{amsart}

\usepackage{amsmath}
\usepackage{amssymb}
\usepackage{graphicx}
\newtheorem{theorem}{Theorem}[section]
\newtheorem{corollary}[theorem]{Corollary}

\theoremstyle{definition}

\newtheorem{remark}[theorem]{Remark}

\numberwithin{equation}{section}

\title[A sharper bound for the Jensen's operator inequality]{A sharper bound for the Jensen's operator inequality}

\author[H. R. Moradi]{Hamid Reza Moradi}

\address[H. R. Moradi]{Department of Mathematics, Payame Noor University (PNU), P.O. Box 19395-4697, Tehran, Iran.}
\email{{\tt hrmoradi@mshdiau.ac.ir}}

\author[S. Furuichi]{Shigeru Furuichi}
\address[S. Furuichi]{Department of Information Science, College of Humanities and Sciences, Nihon University, 3-25-40, Sakurajyousui, Setagaya-ku, Tokyo, 156-8550, Japan.}
\email{\tt furuichi@chs.nihon-u.ac.jp}

\author[M. Sababheh]{Mohammed Sababheh}
\address[M. Sababheh]{Department of Basic Sciences, Princess Sumaya University For Technology,
Al Jubaiha, Amman 11941, Jordan.}
\email{\tt sababheh@psut.edu.jo}

\keywords{Operator inequality, Jensen inequality, convex function, positive linear map.}

\subjclass[2020]{Primary 47A63, Secondary 26D15, 15A60.}

\begin{document}

\begin{abstract}
The primary goal of this paper is to improve the operator version of Jensen inequality. As an application, we provide an improvement for the celebrated Ando's inequality. Additionally, we give a tight bound for the operator H\"older inequality.
\end{abstract}

\maketitle


\section{Introduction}

We begin with fixing some common notations. Let $\mathcal{H}$ be a complex Hilbert space, and $\mathcal{B}\left( \mathcal{H} \right)$ be the algebra of all bounded linear operators on $\mathcal{H}$. An operator $A$ is said to be {\it positive}  (resp. {\it strictly positive}) if and only if $\left\langle Ax,x \right\rangle \ge 0$ for all $x\in \mathcal{H}$ (resp. $\left\langle Ax,x \right\rangle >0$ for all non-zero $x\in \mathcal{H}$). For strictly positive operators $A$ and $B$, the $v$-geometric mean is defined as
\[A{{\sharp}_{v}}B={{A}^{\frac{1}{2}}}{{\left( {{A}^{-\frac{1}{2}}}B{{A}^{-\frac{1}{2}}} \right)}^{v}}{{A}^{\frac{1}{2}}}\quad\text{ }\left( v\in \left[ 0,1 \right] \right).\] 
A real-valued function $f$ defined on an interval $I$ satisfying 
\begin{equation}\label{20}
f\left( \left( 1-v \right)A+vB \right)\le \left( 1-v \right)f\left( A \right)+vf\left( B \right)\quad\text{ }\left( v\in \left[ 0,1 \right] \right)
\end{equation}
for all self-adjoint operators $A,B\in \mathcal{B}\left( \mathcal{H} \right)$ such that $\sigma \left( A \right),\sigma \left( B \right)\subset I$ is called an {\it operator convex} function, where $\sigma \left( X \right)$  means the spectrum of $X\in \mathcal{B}\left( \mathcal{H} \right)$.  The function $f$ is {\it operator concave} on $I$, if the inequality \eqref{20} is reversed. It is an essential fact that $f\left( t \right)={{t}^{r}}$, $r\in \left[ 0,1 \right]$ is operator concave on $\left( 0,\infty  \right)$ and  is operator convex for $r\in \left[ -1,0 \right]\cup \left[ 1,2 \right]$ on $\left( 0,\infty  \right)$.   

  A linear map $\Phi :\mathcal{B}\left( \mathcal{H} \right)\to \mathcal{B}\left( \mathcal{K} \right)$ is called {\it positive} (resp. {\it strictly positive}) if $\Phi \left( A \right)\ge 0$ (resp. $\Phi \left( A \right)>0$) whenever $A\ge 0$ (resp. $A>0$), and $\Phi $ is said to be {\it normalized} if $\Phi \left( {{\mathbf{1}}_{\mathcal{H}}} \right)={{\mathbf{1}}_{\mathcal{K}}}$, where $\mathbf{1}$ is the identity operator.

Let $f:I\to \mathbb{R}$ be a convex function, ${{x}_{1}},\ldots ,{{x}_{n}}\in I$ and ${{w}_{1}},\ldots ,{{w}_{n}}$ be positive numbers with ${{W}_{n}}=\sum\nolimits_{i=1}^{n}{{{w}_{i}}}$. The celebrated Jensen inequality asserts that
\begin{equation}\label{17}
f\left( \frac{1}{{{W}_{n}}}\sum\limits_{i=1}^{n}{{{w}_{i}}{{x}_{i}}} \right)\le \frac{1}{{{W}_{n}}}\sum\limits_{i=1}^{n}{{{w}_{i}}f\left( {{x}_{i}} \right)}.
\end{equation}

 The classical Jensen inequality is one of the essential inequalities in convex analysis, and it has various applications in mathematics, statistics, economics, and engineering sciences. An extensive convex analysis area, including convex functions and their inequalities, is covered in \cite{w1}. The practical applications of convex analysis are also presented in \cite{w2}.

In \cite{3}, one can find an operator form of \eqref{17} which says that if $f:I\to \mathbb{R}$ is an operator convex function, then 
\begin{equation}\label{8}
f\left( \frac{1}{{{W}_{n}}}\sum\limits_{i=1}^{n}{{{w}_{i}}{{A}_{i}}} \right)\le \frac{1}{{{W}_{n}}}\sum\limits_{i=1}^{n}{{{w}_{i}}f\left( {{A}_{i}} \right)}
\end{equation}
whenever ${{A}_{1}},\ldots ,{{A}_{n}}$ are self-adjoint operators with spectra contained in $I$.

The celebrated Choi--Davis--Jensen inequality \cite{5,6} asserts that if $f:I\to \mathbb{R}$ is an operator convex and, $\Phi :\mathcal{B}\left( \mathcal{H} \right)\to \mathcal{B}\left( \mathcal{K} \right)$ is a normalized positive linear mapping, and $A$ is a self-adjoint operator with spectrum contained in $I$, then
\begin{equation}\label{9}
f\left( \Phi \left( A \right) \right)\le \Phi \left( f\left( A \right) \right).
\end{equation}
In the past few years, a considerable attention have been put towards refining or reversing the inequalities \eqref{17}, \eqref{8}, and \eqref{9} and some related inequalities. We refer the interested reader to \cite{01,02,04,03}.

The main result of this paper is included in the next section, where we present an improvement of the operator Jensen inequality inspired by the observation of Dragomir in \cite{dragomir}. This refinement enables us to improve the celebrated Ando's inequality. Additionally, we will refine a known result by Hansen, which is related to the perspective of operator convex functions and positive linear maps. 

\section{Main Results}
As  mentioned in \cite[Corollary 1]{4}, if $f:I\to \mathbb{R}$ is a convex function,  ${{A}_{1}},\ldots ,{{A}_{n}}$ are self-adjoint operators with spectra contained in $I$, and ${{w}_{1}},\ldots ,{{w}_{n}}$ are positive numbers such that $\sum\nolimits_{i=1}^{n}{{{w}_{i}}}=1$, then 
\begin{equation}\label{1}
f\left( \sum\limits_{i=1}^{n}{{{w}_{i}}\left\langle {{A}_{i}}x,x \right\rangle } \right)\le \sum\limits_{i=1}^{n}{{{w}_{i}}\left\langle f\left( {{A}_{i}} \right)x,x \right\rangle }
\end{equation}
where $x\in \mathcal{H}$ with $\left\| x \right\|=1$. 

In the following theorem, we make a refinement of the inequality \eqref{1}.
\begin{theorem}\label{24}
Let $f:I\to \mathbb{R}$ be a convex function,  ${{A}_{1}},\ldots ,{{A}_{n}}$ be self-adjoint operators with spectra contained in $I$, and ${{w}_{1}},\ldots ,{{w}_{n}}$ be positive numbers such that $\sum\nolimits_{i=1}^{n}{{{w}_{i}}}=1$. Assume   $J\subsetneq  \left\{ 1,2,\ldots ,n \right\}$ and  ${{J}^{c}}=\left\{ 1,2,\ldots ,n \right\}\backslash J$,  ${{\omega }_{J}}\equiv \sum\limits_{i\in J}{{{w}_{i}}}$, ${{\omega }_{{{J}^{c}}}}=1-\sum\limits_{i\in J}{{{w}_{i}}}$. Then for any $x\in \mathcal{H}$ with $\left\| x \right\|=1$,
\begin{equation}\label{18}
f\left( \sum\limits_{i=1}^{n}{{{w}_{i}}\left\langle {{A}_{i}}x,x \right\rangle } \right)\le \Psi \left( f,\mathbb{A},J,{{J}^{c}} \right)\le \sum\limits_{i=1}^{n}{{{w}_{i}}\left\langle f\left( {{A}_{i}} \right)x,x \right\rangle }
\end{equation}
where
\[\Psi \left( f,\mathbb{A},J,{{J}^{c}} \right)\equiv {{\omega }_{J}}f\left( \frac{1}{{{\omega }_{J}}}\sum\limits_{i\in J}{{{w}_{i}}\left\langle {{A}_{i}}x,x \right\rangle } \right)+{{\omega }_{{{J}^{c}}}}f\left( \frac{1}{{{\omega }_{{{J}^{c}}}}}\sum\limits_{i\in {{J}^{c}}}{{{w}_{i}}\left\langle {{A}_{i}}x,x \right\rangle } \right).\]
The inequality \eqref{18} is reversed if the function $f$ is concave on $I$.
\end{theorem}
\begin{proof}
We can replace ${{x}_{i}}$ by $\left\langle {{A}_{i}}x,x \right\rangle $ where $x\in \mathcal{H}$ and $\left\| x \right\|=1$, in \eqref{17}. Hence, by using \cite[Theorem 1.2]{1}, we can immediately infer that
\begin{equation}\label{7}
\begin{aligned}
f\left( \frac{1}{{{W}_{n}}}\sum\limits_{i=1}^{n}{{{w}_{i}}\left\langle {{A}_{i}}x,x \right\rangle } \right)&\le \frac{1}{{{W}_{n}}}\sum\limits_{i=1}^{n}{{{w}_{i}}f\left( \left\langle {{A}_{i}}x,x \right\rangle  \right)} \\ 
& \le \frac{1}{{{W}_{n}}}\sum\limits_{i=1}^{n}{{{w}_{i}}\left\langle f\left( {{A}_{i}} \right)x,x \right\rangle } 
\end{aligned}
\end{equation}
where ${{W}_{n}}=\sum\nolimits_{i=1}^{n}{{{w}_{i}}}$. Now, a simple calculation shows that
\begin{equation}\label{3}
\begin{aligned}
\sum\limits_{i=1}^{n}{{{w}_{i}}\left\langle f\left( {{A}_{i}} \right)x,x \right\rangle }&=\sum\limits_{i\in J}{{{w}_{i}}\left\langle f\left( {{A}_{i}} \right)x,x \right\rangle }+\sum\limits_{i\in {{J}^{c}}}{{{w}_{i}}\left\langle f\left( {{A}_{i}} \right)x,x \right\rangle } \\ 
& ={{\omega }_{J}}\left( \frac{1}{{{\omega }_{J}}}\sum\limits_{i\in J}{{{w}_{i}}\left\langle f\left( {{A}_{i}} \right)x,x \right\rangle } \right)+{{\omega }_{{{J}^{c}}}}\left( \frac{1}{{{\omega }_{{{J}^{c}}}}}\sum\limits_{i\in {{J}^{c}}}{{{w}_{i}}\left\langle f\left( {{A}_{i}} \right)x,x \right\rangle } \right) \\ 
& \ge {{\omega }_{J}}f\left( \frac{1}{{{\omega }_{J}}}\sum\limits_{i\in J}{{{w}_{i}}\left\langle {{A}_{i}}x,x \right\rangle } \right)+{{\omega }_{{{J}^{c}}}}f\left( \frac{1}{{{\omega }_{{{J}^{c}}}}}\sum\limits_{i\in {{J}^{c}}}{{{w}_{i}}\left\langle {{A}_{i}}x,x \right\rangle } \right) \\ 
& =\Psi \left( f,\mathbb{A},J,{{J}^{c}} \right)  
\end{aligned}
\end{equation}
where we used the inequality \eqref{7}. On the other hand,
\begin{eqnarray}\label{2}
 \Psi \left( f,\mathbb{A},J,{{J}^{c}} \right)&=&{{\omega }_{J}}f\left( \frac{1}{{{\omega }_{J}}}\sum\limits_{i\in J}{{{w}_{i}}\left\langle {{A}_{i}}x,x \right\rangle } \right)+{{\omega }_{{{J}^{c}}}}f\left( \frac{1}{{{\omega }_{{{J}^{c}}}}}\sum\limits_{i\in {{J}^{c}}}{{{w}_{i}}\left\langle {{A}_{i}}x,x \right\rangle } \right) \nonumber \\ 
& \ge& f\left( {{\omega }_{J}}\left( \frac{1}{{{\omega }_{J}}}\sum\limits_{i\in J}{{{w}_{i}}\left\langle {{A}_{i}}x,x \right\rangle } \right)+{{\omega }_{{{J}^{c}}}}\left( \frac{1}{{{\omega }_{{{J}^{c}}}}}\sum\limits_{i\in {{J}^{c}}}{{{w}_{i}}\left\langle {{A}_{i}}x,x \right\rangle } \right) \right)\nonumber \\ 
& =&f\left( \sum\limits_{i=1}^{n}{{{w}_{i}}\left\langle {{A}_{i}}x,x \right\rangle } \right).  
\end{eqnarray}
In the above computations we have used the assumption that $f$ is a convex function.

Now \eqref{3} together with inequality \eqref{2} yield the inequality \eqref{18}.
\end{proof}

The following refinements of the arithmetic--geometric--harmonic  mean inequality are of interest.
\begin{corollary}
	Let ${{a}_{1}},\ldots ,{{a}_{n}}$ be positive numbers and let $\left\{ {{w}_{i}} \right\},J,{{J}^{c}}$ be as in Theorem \ref{24}. Then
\begin{eqnarray*}
&& {{\left( \sum\limits_{i=1}^{n}{{{w}_{i}}a_{i}^{-1}} \right)}^{-1}}\le {{\left( \frac{1}{{{\omega }_{J}}}\sum\limits_{i\in J}{{{w}_{i}}a_{i}^{-1}} \right)}^{-{{\omega }_{J}}}}{{\left( \frac{1}{{{\omega }_{{{J}^{c}}}}}\sum\limits_{i\in {{J}^{c}}}{{{w}_{i}}a_{i}^{-1}} \right)}^{-{{\omega }_{{{J}^{c}}}}}} \\ 
&& \le \prod\limits_{i=1}^{n}{a_{i}^{{{w}_{i}}}}  \le {{\left( \frac{1}{{{\omega }_{J}}}\sum\limits_{i\in J}{{{w}_{i}}{{a}_{i}}} \right)}^{{{\omega }_{J}}}}{{\left( \frac{1}{{{\omega }_{{{J}^{c}}}}}\sum\limits_{i\in {{J}^{c}}}{{{w}_{i}}{{a}_{i}}} \right)}^{{{\omega }_{{{J}^{c}}}}}} 
\le \sum\limits_{i=1}^{n}{{{w}_{i}}{{a}_{i}}}  
\end{eqnarray*}
and
\begin{eqnarray*}
&& {{\left( \sum\limits_{i=1}^{n}{{{w}_{i}}a_{i}^{-1}} \right)}^{-1}}\le {{\left( {{\omega }_{J}}\prod\limits_{i\in J}{a_{i}^{-\frac{{{w}_{i}}}{{{\omega }_{J}}}}}+{{\omega }_{{{J}^{c}}}}\prod\limits_{i\in {{J}^{c}}}{a_{i}^{-\frac{{{w}_{i}}}{{{\omega }_{{{J}^{c}}}}}}} \right)}^{-1}} \\ 
&& \le \prod\limits_{i=1}^{n}{a_{i}^{{{w}_{i}}}} \le {{\omega }_{J}}\prod\limits_{i\in J}{a_{i}^{\frac{{{w}_{i}}}{{{\omega }_{J}}}}}+{{\omega }_{{{J}^{c}}}}\prod\limits_{i\in {{J}^{c}}}{a_{i}^{\frac{{{w}_{i}}}{{{\omega }_{{{J}^{c}}}}}}} \le \sum\limits_{i=1}^{n}{{{w}_{i}}{{a}_{i}}}.  
\end{eqnarray*}
\end{corollary}

By virtue of Theorem \ref{24}, we have the following result:
\begin{corollary}\label{16}
	Let $f:I\to \mathbb{R}$ be a non-negative increasing convex function,  ${{A}_{1}},\ldots ,{{A}_{n}}$ be positive operators with spectra contained in $I$, and let $\left\{ {{w}_{i}} \right\},J,{{J}^{c}}$ be as in Theorem \ref{24}. Then
	\begin{eqnarray}\label{11}
f\left( \left\| \sum\limits_{i=1}^{n}{{{w}_{i}}{{A}_{i}}} \right\| \right)&\le& {{\omega }_{J}}f\left( \frac{1}{{{\omega }_{J}}}\left\| \sum\limits_{i=1}^{n}{{{w}_{i}}{{A}_{i}}} \right\| \right)+{{\omega }_{{{J}^{c}}}}f\left( \frac{1}{{{\omega }_{{{J}^{c}}}}}\left\| \sum\limits_{i=1}^{n}{{{w}_{i}}{{A}_{i}}} \right\| \right)\nonumber\\
&\le& \left\| \sum\limits_{i=1}^{n}{{{w}_{i}}f\left( {{A}_{i}} \right)} \right\|.
	\end{eqnarray}
	The inequality \eqref{11} is reversed if the function $f$ is non-negative increasing concave on $I$.
\end{corollary}
\begin{proof}
On account of the assumptions, we have
\begin{eqnarray*}
\underset{\left\| x \right\|=1}{\mathop{\sup }}\,f\left( \sum\limits_{i=1}^{n}{{{w}_{i}}\left\langle {{A}_{i}}x,x \right\rangle } \right)
&=&f\left( \underset{\left\| x \right\|=1}{\mathop{\sup }}\,\left\langle \sum\limits_{i=1}^{n}{{{w}_{i}}{{A}_{i}}}x,x \right\rangle  \right)  
=f\left( \left\| \sum\limits_{i=1}^{n}{{{w}_{i}}{{A}_{i}}} \right\| \right) \\ 
&\le& {{\omega }_{J}}f\left( \frac{1}{{{\omega }_{J}}}\left\| \sum\limits_{i\in J}{{{w}_{i}}{{A}_{i}}} \right\| \right)+{{\omega }_{{{J}^{c}}}}f\left( \frac{1}{{{\omega }_{{{J}^{c}}}}}\left\| \sum\limits_{i\in {{J}^{c}}}{{{w}_{i}}{{A}_{i}}} \right\| \right) \\ 
&\le& \underset{\left\| x \right\|=1}{\mathop{\sup }}\,\left\langle \sum\limits_{i=1}^{n}{{{w}_{i}}f\left( {{A}_{i}} \right)}x,x \right\rangle   
=\left\| \sum\limits_{i=1}^{n}{{{w}_{i}}f\left( {{A}_{i}} \right)} \right\|.  
\end{eqnarray*}
This completes the proof.
\end{proof}

The following remark is worth mentioning.
\begin{remark}
Let  ${{A}_{1}},\ldots ,{{A}_{n}}$ be positive operators and let $\left\{ {{w}_{i}} \right\},J,{{J}^{c}}$ be as in Theorem \ref{24}. Then for any $r\ge 1$,
\begin{eqnarray}\label{12}
{{\left\| \sum\limits_{i=1}^{n}{{{w}_{i}}{{A}_{i}}} \right\|}^{r}}&\le& {{\omega }_{J}}{{\left( \frac{1}{{{\omega }_{J}}}\left\| \sum\limits_{i=1}^{n}{{{w}_{i}}{{A}_{i}}} \right\| \right)}^{r}}+{{\omega }_{{{J}^{c}}}}{{\left( \frac{1}{{{\omega }_{{{J}^{c}}}}}\left\| \sum\limits_{i=1}^{n}{{{w}_{i}}{{A}_{i}}} \right\| \right)}^{r}}\nonumber \\
&\le& \left\| \sum\limits_{i=1}^{n}{{{w}_{i}}A_{i}^{r}} \right\|.
\end{eqnarray}
For $0<r\le 1$, the reverse inequalities hold. If the operators are strictly positive, then the above inequality is also true for $r<0$.
\end{remark}

The multiple version of the inequality \eqref{9} is proved in \cite[Theorem 1]{2} as follows: Let $f:I\to \mathbb{R}$ be an operator convex function,  ${{\Phi }_{1}},\ldots ,{{\Phi }_{n}}$ be normalized positive linear mappings from $\mathcal{B}\left( \mathcal{H} \right)$ to $\mathcal{B}\left( \mathcal{K} \right)$, ${{A}_{1}},\ldots ,{{A}_{n}}$ be self-adjoint operators with spectra contained in $I$, and ${{w}_{1}},\ldots ,{{w}_{n}}$ be  positive numbers such that $\sum\nolimits_{i=1}^{n}{{{w}_{i}}}=1$, then
\begin{equation}\label{10}
f\left( \sum\limits_{i=1}^{n}{{{w}_{i}}{{\Phi }_{i}}\left( {{A}_{i}} \right)} \right)\le \sum\limits_{i=1}^{n}{{{w}_{i}}f\left( {{\Phi }_{i}}\left( {{A}_{i}} \right) \right)}.
\end{equation}

The following is a refinement of \eqref{10}. This result was found by Moslehian and Kian \cite[Corollary 3.2]{mok}, with a different expression. However, we mimic some ideas of Dragomir \cite[Theorem 1]{dragomir} to obtain it.
\begin{theorem}\label{13}
	Let $f:I\to \mathbb{R}$ be an operator convex function,  ${{\Phi }_{1}},\ldots ,{{\Phi }_{n}}$ be normalized positive linear mappings from $\mathcal{B}\left( \mathcal{H} \right)$ to $\mathcal{B}\left( \mathcal{K} \right)$, ${{A}_{1}},\ldots ,{{A}_{n}}$ be self-adjoint operators with spectra contained in $I$, and let $\left\{ {{w}_{i}} \right\},J,{{J}^{c}}$ be as in Theorem \ref{24}. Then
	\begin{equation}\label{19}
f\left( \sum\limits_{i=1}^{n}{{{w}_{i}}{{\Phi }_{i}}\left( {{A}_{i}} \right)} \right)\le \Delta \left( f,\mathbb{A},J,{{J}^{c}} \right)\le \sum\limits_{i=1}^{n}{{{w}_{i}}{{\Phi }_{i}}\left( f\left( {{A}_{i}} \right) \right)}
	\end{equation}
where
\[\Delta \left( f,\mathbb{A},J,{{J}^{c}} \right)\equiv {{\omega }_{J}}f\left( \frac{1}{{{\omega }_{J}}}\sum\limits_{i\in J}{{{w}_{i}}{{\Phi }_{i}}\left( {{A}_{i}} \right)} \right)+{{\omega }_{{{J}^{c}}}}f\left( \frac{1}{{{\omega }_{{{J}^{c}}}}}\sum\limits_{i\in {{J}^{c}}}{{{w}_{i}}{{\Phi }_{i}}\left( {{A}_{i}} \right)} \right).\]
The inequality \eqref{19} reversed if the function $f$ is operator concave on $I$.
\end{theorem}
\begin{proof}
It can be easily shown that
		\begin{equation}\label{4}
	f\left( \frac{1}{{{W}_{n}}}\sum\limits_{i=1}^{n}{{{w}_{i}}{{\Phi }_{i}}\left( {{A}_{i}} \right)} \right)\le \frac{1}{{{W}_{n}}}\sum\limits_{i=1}^{n}{{{w}_{i}}{{\Phi }_{i}}\left( f\left( {{A}_{i}} \right) \right)}	
	\end{equation}
	where ${{W}_{n}}=\sum\nolimits_{i=1}^{n}{{{w}_{i}}}$. By employing the inequality \eqref{4} we have
\begin{eqnarray}\label{6}
\sum\limits_{i=1}^{n}{{{w}_{i}}{{\Phi }_{i}}\left( f\left( {{A}_{i}} \right) \right)}&=&\sum\limits_{i\in J}{{{w}_{i}}{{\Phi }_{i}}\left( f\left( {{A}_{i}} \right) \right)}+\sum\limits_{i\in {{J}^{c}}}{{{w}_{i}}{{\Phi }_{i}}\left( f\left( {{A}_{i}} \right) \right)} \nonumber\\ 
& =&{{\omega }_{J}}\left( \frac{1}{{{\omega }_{J}}}\sum\limits_{i\in J}{{{w}_{i}}{{\Phi }_{i}}\left( f\left( {{A}_{i}} \right) \right)} \right)+{{\omega }_{{{J}^{c}}}}\left( \frac{1}{{{\omega }_{{{J}^{c}}}}}\sum\limits_{i\in {{J}^{c}}}{{{w}_{i}}{{\Phi }_{i}}\left( f\left( {{A}_{i}} \right) \right)} \right) \nonumber\\ 
& \ge& {{\omega }_{J}}f\left( \frac{1}{{{\omega }_{J}}}\sum\limits_{i\in J}{{{w}_{i}}{{\Phi }_{i}}\left( {{A}_{i}} \right)} \right)+{{\omega }_{{{J}^{c}}}}f\left( \frac{1}{{{\omega }_{{{J}^{c}}}}}\sum\limits_{i\in {{J}^{c}}}{{{w}_{i}}{{\Phi }_{i}}\left( {{A}_{i}} \right)} \right) \nonumber \\ 
& =&\Delta \left( f,\mathbb{A},J,{{J}^{c}} \right).  
\end{eqnarray}	
On the other hand, since $f$ is an operator convex function, we get
\begin{eqnarray}\label{5}
\Delta \left( f,\mathbb{A},J,{{J}^{c}} \right)&=&{{\omega }_{J}}f\left( \frac{1}{{{\omega }_{J}}}\sum\limits_{i\in J}{{{w}_{i}}{{\Phi }_{i}}\left( {{A}_{i}} \right)} \right)+{{\omega }_{{{J}^{c}}}}f\left( \frac{1}{{{\omega }_{{{J}^{c}}}}}\sum\limits_{i\in {{J}^{c}}}{{{w}_{i}}{{\Phi }_{i}}\left( {{A}_{i}} \right)} \right) \nonumber\\ 
& \ge& f\left( {{\omega }_{J}}\left( \frac{1}{{{\omega }_{J}}}\sum\limits_{i\in J}{{{w}_{i}}{{\Phi }_{i}}\left( {{A}_{i}} \right)} \right)+{{\omega }_{{{J}^{c}}}}\left( \frac{1}{{{\omega }_{{{J}^{c}}}}}\sum\limits_{i\in {{J}^{c}}}{{{w}_{i}}{{\Phi }_{i}}\left( {{A}_{i}} \right)} \right) \right) \nonumber \\ 
& =&f\left( \sum\limits_{i=1}^{n}{{{w}_{i}}{{\Phi }_{i}}\left( {{A}_{i}} \right)} \right).  
\end{eqnarray}
Combining the two inequalities \eqref{6} and \eqref{5}, we have the desired inequality.
\end{proof}

A special case of \eqref{19} is the following statement:
\begin{remark}\label{21}
		Let ${{\Phi }_{1}},\ldots ,{{\Phi }_{n}}$ be normalized positive linear mappings from $\mathcal{B}\left( \mathcal{H} \right)$ to $\mathcal{B}\left( \mathcal{K} \right)$, ${{A}_{1}},\ldots ,{{A}_{n}}$ be self-adjoint operators with spectra contained in $I$, and let $\left\{ {{w}_{i}} \right\},J,{{J}^{c}}$ be as in Theorem \ref{24}. Then for any $r\in \left[ -1,0 \right]\cup \left[ 1,2 \right]$,
\begin{eqnarray*}
{{\left( \sum\limits_{i=1}^{n}{{{w}_{i}}{{\Phi }_{i}}\left( {{A}_{i}} \right)} \right)}^{r}}&\le& {{\omega }_{J}}{{\left( \frac{1}{{{\omega }_{J}}}\sum\limits_{i\in J}{{{w}_{i}}{{\Phi }_{i}}\left( {{A}_{i}} \right)} \right)}^{r}}+{{\omega }_{{{J}^{c}}}}{{\left( \frac{1}{{{\omega }_{{{J}^{c}}}}}\sum\limits_{i\in {{J}^{c}}}{{{w}_{i}}{{\Phi }_{i}}\left( {{A}_{i}} \right)} \right)}^{r}}\\
&\le& \sum\limits_{i=1}^{n}{{{w}_{i}}{{\Phi }_{i}}\left( A_{i}^{r} \right)}.
\end{eqnarray*}
For $r\in \left[ 0,1 \right]$, the reverse inequalities hold.
\end{remark}

The next corollary can be compared to \cite[Theorem 1]{hirza}.
\begin{corollary}
		Let $\mathcal{H}$ and $\mathcal{K}$ be finite dimensional  Hilbert spaces, ${{\Phi }_{1}},\ldots ,{{\Phi }_{n}}$ be normalized positive linear mappings from $\mathcal{B}\left( \mathcal{H} \right)$ to $\mathcal{B}\left( \mathcal{K} \right)$, ${{A}_{1}},\ldots ,{{A}_{n}}$ be self-adjoint operators with spectra contained in $I$, and let $\left\{ {{w}_{i}} \right\},J,{{J}^{c}}$ be as in Theorem \ref{24}. Then for any  $r\ge 1$ and every unitarily invariant norm $\| \cdot \|_u$,
\begin{eqnarray}\label{26}
&&\left\| {{\left( \sum\limits_{i=1}^{n}{{{w}_{i}}{{\Phi }_{i}}\left( {{A}_{i}} \right)} \right)}^{r}} \right\|_u\nonumber \\
&&\le \left\| {{\left( {{\omega }_{J}}{{\left( \frac{1}{{{\omega }_{J}}}\sum\limits_{i\in J}{{{w}_{i}}{{\Phi }_{i}}\left( A_{i}^{r} \right)} \right)}^{\frac{1}{r}}}+{{\omega }_{{{J}^{c}}}}{{\left( \frac{1}{{{\omega }_{{{J}^{c}}}}}\sum\limits_{i\in {{J}^{c}}}{{{w}_{i}}{{\Phi }_{i}}\left( A_{i}^{r} \right)} \right)}^{\frac{1}{r}}} \right)}^{r}} \right\|_u  \\ 
&& \le \left\| \sum\limits_{i=1}^{n}{{{w}_{i}}{{\Phi }_{i}}\left( A_{i}^{r} \right)} \right\|_u.  \nonumber
\end{eqnarray}
In particular,
\begin{eqnarray}\label{27}
&&\left\| {{\left( \sum\limits_{i=1}^{n}{{{w}_{i}}X_{i}^{*}{{A}_{i}}{{X}_{i}}} \right)}^{r}} \right\|_u\nonumber \\
&&\le \left\| {{\left( {{\omega }_{J}}{{\left( \frac{1}{{{\omega }_{J}}}\sum\limits_{i\in J}{{{w}_{i}}X_{i}^{*}A_{i}^{r}{{X}_{i}}} \right)}^{\frac{1}{r}}}+{{\omega }_{{{J}^{c}}}}{{\left( \frac{1}{{{\omega }_{{{J}^{c}}}}}\sum\limits_{i\in {{J}^{c}}}{{{w}_{i}}X_{i}^{*}A_{i}^{r}{{X}_{i}}} \right)}^{\frac{1}{r}}} \right)}^{r}} \right\|_u  \\ 
&& \le \left\| \sum\limits_{i=1}^{n}{{{w}_{i}}X_{i}^{*}A_{i}^{r}{{X}_{i}}} \right\|_u  \nonumber
\end{eqnarray}
where each $X_i$ ($i=1,2,\cdots,n$) is an isometry.
\end{corollary}
\begin{proof}
Of course, the inequality \eqref{27} is a direct consequence of inequality \eqref{26}, so we prove \eqref{26}. It follows from Remark \ref{21} that  	
\begin{eqnarray*}
 \left\| \sum\limits_{i=1}^{n}{{{w}_{i}}{{\Phi }_{i}}\left( A_{i}^{\frac{1}{r}} \right)} \right\|_u&\le& \left\| {{\omega }_{J}}{{\left( \frac{1}{{{\omega }_{J}}}\sum\limits_{i\in J}{{{w}_{i}}{{\Phi }_{i}}\left( {{A}_{i}} \right)} \right)}^{\frac{1}{r}}}+{{\omega }_{{{J}^{c}}}}{{\left( \frac{1}{{{\omega }_{{{J}^{c}}}}}\sum\limits_{i\in {{J}^{c}}}{{{w}_{i}}{{\Phi }_{i}}\left( {{A}_{i}} \right)} \right)}^{\frac{1}{r}}} \right\|_u \\ 
& \le& \left\| {{\left( \sum\limits_{i=1}^{n}{{{w}_{i}}{{\Phi }_{i}}\left( {{A}_{i}} \right)} \right)}^{\frac{1}{r}}} \right\|_u  
\end{eqnarray*}
for any $r\ge 1$. Replacing ${{A}_{i}}$ by $A_{i}^{r}$, we get 
\begin{eqnarray}\label{25}
\left\| \sum\limits_{i=1}^{n}{{{w}_{i}}{{\Phi }_{i}}\left( {{A}_{i}} \right)} \right\|_u&\le& \left\| {{\omega }_{J}}{{\left( \frac{1}{{{\omega }_{J}}}\sum\limits_{i\in J}{{{w}_{i}}{{\Phi }_{i}}\left( A_{i}^{r} \right)} \right)}^{\frac{1}{r}}}+{{\omega }_{{{J}^{c}}}}{{\left( \frac{1}{{{\omega }_{{{J}^{c}}}}}\sum\limits_{i\in {{J}^{c}}}{{{w}_{i}}{{\Phi }_{i}}\left( A_{i}^{r} \right)} \right)}^{\frac{1}{r}}} \right\|_u\nonumber  \\ 
& \le & \left\| {{\left( \sum\limits_{i=1}^{n}{{{w}_{i}}{{\Phi }_{i}}\left( A_{i}^{r} \right)} \right)}^{\frac{1}{r}}} \right\|_u.  
\end{eqnarray}
It is well-known that ${{\left\| X \right\|}_{r}}={{\left\| {{\left| X \right|}^{r}} \right\|}^{\frac{1}{r}}}$ defines a unitarily invariant norm. So \eqref{25} implies
\[\begin{aligned}
 \left\| {{\left( \sum\limits_{i=1}^{n}{{{w}_{i}}{{\Phi }_{i}}\left( {{A}_{i}} \right)} \right)}^{r}} \right\|_u&\le \left\| {{\left( {{\omega }_{J}}{{\left( \frac{1}{{{\omega }_{J}}}\sum\limits_{i\in J}{{{w}_{i}}{{\Phi }_{i}}\left( A_{i}^{r} \right)} \right)}^{\frac{1}{r}}}+{{\omega }_{{{J}^{c}}}}{{\left( \frac{1}{{{\omega }_{{{J}^{c}}}}}\sum\limits_{i\in {{J}^{c}}}{{{w}_{i}}{{\Phi }_{i}}\left( A_{i}^{r} \right)} \right)}^{\frac{1}{r}}} \right)}^{r}} \right\|_u \\ 
& \le \left\| \sum\limits_{i=1}^{n}{{{w}_{i}}{{\Phi }_{i}}\left( A_{i}^{r} \right)} \right\|_u.  
\end{aligned}\]
The proof is complete.
\end{proof}
Kubo and Ando \cite{k} showed that for every operator mean $\sigma $ there exists an operator monotone function $f:\left( 0,\infty  \right)\to \left( 0,\infty  \right)$ such that
\begin{equation}\label{23}
A\sigma B={{A}^{\frac{1}{2}}}f\left( {{A}^{\frac{-1}{2}}}{{B}^{-1}}{{A}^{\frac{-1}{2}}} \right){{A}^{\frac{1}{2}}}
\end{equation}
for all $A,B>0$. They also proved that if $f:\left( 0,\infty  \right)\to \left( 0,\infty  \right)$ is operator monotone, the binary operation defined by \eqref{23} is an operator mean.

We know that (see the estimate (16) in \cite{kr}) if $\sigma $ is an operator mean (in the Kubo-Ando sense) and ${{A}_{i}},{{B}_{i}}>0$, then
\begin{equation}\label{22}
\sum\limits_{i=1}^{n}{{{w}_{i}}\left( {{A}_{i}}\sigma {{B}_{i}} \right)}\le \left( \sum\limits_{i=1}^{n}{{{w}_{i}}{{A}_{i}}} \right)\sigma \left( \sum\limits_{i=1}^{n}{{{w}_{i}}{{B}_{i}}} \right).
\end{equation}
The following corollary can be regarded as a refinement and generalization of the inequality \eqref{22}.
\begin{corollary}\label{15}
		Let $\sigma $ be an operator mean,  ${{\Phi }_{1}},\ldots ,{{\Phi }_{n}}$ be normalized positive linear mappings from $\mathcal{B}\left( \mathcal{H} \right)$ to $\mathcal{B}\left( \mathcal{K} \right)$, ${{A}_{1}},\ldots ,{{A}_{n}}$, ${{B}_{1}},\ldots ,{{B}_{n}}$ be strictly positive operators with spectra contained in $I$, and let $\left\{ {{w}_{i}} \right\},J,{{J}^{c}}$ be as in Theorem \ref{24}. Then
\[\begin{aligned}
& \sum\limits_{i=1}^{n}{{{w}_{i}}{{\Phi }_{i}}\left( {{A}_{i}}\sigma {{B}_{i}} \right)} \\ 
& \le \left( \sum\limits_{i\in J}{{{w}_{i}}{{\Phi }_{i}}\left( {{A}_{i}} \right)} \right)\sigma \left( \sum\limits_{i\in J}{{{w}_{i}}{{\Phi }_{i}}\left( {{B}_{i}} \right)} \right)+\left( \sum\limits_{i\in {{J}^{c}}}{{{w}_{i}}{{\Phi }_{i}}\left( {{A}_{i}} \right)} \right)\sigma \left( \sum\limits_{i\in {{J}^{c}}}{{{w}_{i}}{{\Phi }_{i}}\left( {{B}_{i}} \right)} \right) \\ 
& \le \left( \sum\limits_{i=1}^{n}{{{w}_{i}}{{\Phi }_{i}}\left( {{A}_{i}} \right)} \right)\sigma \left( \sum\limits_{i=1}^{n}{{{w}_{i}}{{\Phi }_{i}}\left( {{B}_{i}} \right)} \right).
\end{aligned}\]
\end{corollary}
\begin{proof}
If $F\left( \cdot,\cdot \right)$ is a jointly operator concave function, then Theorem \ref{13} implies
\begin{eqnarray}\label{14}
&& \sum\limits_{i=1}^{n}{{{w}_{i}}{{\Phi }_{i}}\left( F\left( {{A}_{i}},{{B}_{i}} \right) \right)} \le {{\omega }_{J}}F\left( \frac{1}{{{\omega }_{J}}}\sum\limits_{i\in J}{{{w}_{i}}{{\Phi }_{i}}\left( {{A}_{i}} \right)},\frac{1}{{{\omega }_{J}}}\sum\limits_{i\in J}{{{w}_{i}}{{\Phi }_{i}}\left( {{B}_{i}} \right)} \right) \nonumber \\ 
&& +{{\omega }_{{{J}^{c}}}}F\left( \frac{1}{{{\omega }_{{{J}^{c}}}}}\sum\limits_{i\in {{J}^{c}}}{{{w}_{i}}{{\Phi }_{i}}\left( {{A}_{i}} \right)},\frac{1}{{{\omega }_{{{J}^{c}}}}}\sum\limits_{i\in {{J}^{c}}}{{{w}_{i}}{{\Phi }_{i}}\left( {{B}_{i}} \right)} \right) \nonumber \\
&& \le F\left( \sum\limits_{i=1}^{n}{{{w}_{i}}{{\Phi }_{i}}\left( {{A}_{i}} \right)},\sum\limits_{i=1}^{n}{{{w}_{i}}{{\Phi }_{i}}\left( {{B}_{i}} \right)} \right).
\end{eqnarray}
It is well-known that $F\left( A,B \right)=A\sigma B$ is jointly concave \cite{ando}, so it follows from \eqref{14} that
\begin{eqnarray*}
&& \sum\limits_{i=1}^{n}{{{w}_{i}}{{\Phi }_{i}}\left( {{A}_{i}}\sigma {{B}_{i}} \right)}  \le {{\omega }_{J}}\left( \left( \frac{1}{{{\omega }_{J}}}\sum\limits_{i\in J}{{{w}_{i}}{{\Phi }_{i}}\left( {{A}_{i}} \right)} \right)\sigma \left( \frac{1}{{{\omega }_{J}}}\sum\limits_{i\in J}{{{w}_{i}}{{\Phi }_{i}}\left( {{B}_{i}} \right)} \right) \right) \\ 
&&\qquad +{{\omega }_{{{J}^{c}}}}\left( \left( \frac{1}{{{\omega }_{{{J}^{c}}}}}\sum\limits_{i\in {{J}^{c}}}{{{w}_{i}}{{\Phi }_{i}}\left( {{A}_{i}} \right)} \right)\sigma \left( \frac{1}{{{\omega }_{{{J}^{c}}}}}\sum\limits_{i\in {{J}^{c}}}{{{w}_{i}}{{\Phi }_{i}}\left( {{B}_{i}} \right)} \right) \right) \\ 
&& =\left( \sum\limits_{i\in J}{{{w}_{i}}{{\Phi }_{i}}\left( {{A}_{i}} \right)} \right)\sigma \left( \sum\limits_{i\in J}{{{w}_{i}}{{\Phi }_{i}}\left( {{B}_{i}} \right)} \right)+\left( \sum\limits_{i\in {{J}^{c}}}{{{w}_{i}}{{\Phi }_{i}}\left( {{A}_{i}} \right)} \right)\sigma \left( \sum\limits_{i\in {{J}^{c}}}{{{w}_{i}}{{\Phi }_{i}}\left( {{B}_{i}} \right)} \right) \\ 
&& \le \left( \sum\limits_{i=1}^{n}{{{w}_{i}}{{\Phi }_{i}}\left( {{A}_{i}} \right)} \right)\sigma \left( \sum\limits_{i=1}^{n}{{{w}_{i}}{{\Phi }_{i}}\left( {{B}_{i}} \right)} \right),
\end{eqnarray*}
thanks to the homogeneity property of operator means. This completes the proof.
\end{proof}

By setting $\sigma ={{\sharp}_{v}}\left( v\in \left[ 0,1 \right] \right)$ and ${{\Phi }_{i}}\left( {{X}_{i}} \right)={{X}_{i}}\left( i=1,\ldots ,n \right)$ in Corollary \ref{15}, we improve the weighted operator H\"older and Cauchy inequalities in the following way:
\begin{corollary}
			Let ${{\Phi }_{1}},\ldots ,{{\Phi }_{n}}$ be normalized positive linear mappings from $\mathcal{B}\left( \mathcal{H} \right)$ to $\mathcal{B}\left( \mathcal{K} \right)$, ${{A}_{1}},\ldots ,{{A}_{n}}$, ${{B}_{1}},\ldots ,{{B}_{n}}$ be strictly positive operators with spectra contained in $I$, and let $\left\{ {{w}_{i}} \right\},J,{{J}^{c}}$ be as in Theorem \ref{24}. Then for any $v\in \left[ 0,1 \right]$,
\[\begin{aligned}
 \sum\limits_{i=1}^{n}{{{w}_{i}}\left( {{A}_{i}}{{\sharp}_{v}}{{B}_{i}} \right)}&\le \left( \sum\limits_{i\in J}{{{w}_{i}}{{A}_{i}}} \right){{\sharp}_{v}}\left( \sum\limits_{i\in J}{{{w}_{i}}{{B}_{i}}} \right)+\left( \sum\limits_{i\in {{J}^{c}}}{{{w}_{i}}{{A}_{i}}} \right){{\sharp}_{v}}\left( \sum\limits_{i\in {{J}^{c}}}{{{w}_{i}}{{B}_{i}}} \right) \\ 
& \le \left( \sum\limits_{i=1}^{n}{{{w}_{i}}{{A}_{i}}} \right){{\sharp}_{v}}\left( \sum\limits_{i=1}^{n}{{{w}_{i}}{{B}_{i}}} \right).  
\end{aligned}\]
In particular,
\[\begin{aligned}
 \sum\limits_{i=1}^{n}{{{w}_{i}}\left( {{A}_{i}}\sharp{{B}_{i}} \right)}&\le \left( \sum\limits_{i\in J}{{{w}_{i}}{{A}_{i}}} \right)\sharp\left( \sum\limits_{i\in J}{{{w}_{i}}{{B}_{i}}} \right)+\left( \sum\limits_{i\in {{J}^{c}}}{{{w}_{i}}{{A}_{i}}} \right)\sharp\left( \sum\limits_{i\in {{J}^{c}}}{{{w}_{i}}{{B}_{i}}} \right) \\ 
& \le \left( \sum\limits_{i=1}^{n}{{{w}_{i}}{{A}_{i}}} \right)\sharp\left( \sum\limits_{i=1}^{n}{{{w}_{i}}{{B}_{i}}} \right).  
\end{aligned}\]
\end{corollary}
Recall that if $f$ is operator convex, then \eqref{23} defines \cite{niku} the perspective of $f$ denoted by ${{\mathcal{P}}_{f}}\left( A\mid B \right)$, i.e.,
\[{{\mathcal{P}}_{f}}\left( A\mid B \right)={{A}^{\frac{1}{2}}}f\left( {{A}^{-\frac{1}{2}}}B{{A}^{-\frac{1}{2}}} \right){{A}^{\frac{1}{2}}}.\]
The operator perspective enjoys the following property:
\[{{\mathcal{P}}_{f}}\left( \Phi \left( A \right)\mid\Phi \left( B \right) \right)\le \Phi \left( {{\mathcal{P}}_{f}}\left( A\mid B \right) \right).\]
This nice inequality has been proved by Hansen \cite{g,han}.
Let us note that the perspective of an operator convex function is operator convex
as a function of two variables (see \cite[Theorem 2.2]{niku}). 

So, taking into account the above and applying Theorem \ref{13}, we get the following result.
\begin{corollary}
		Let $f:I\to \mathbb{R}$ be an operator convex function,  ${{\Phi }_{1}},\ldots ,{{\Phi }_{n}}$ be normalized positive linear mappings from $\mathcal{B}\left( \mathcal{H} \right)$ to $\mathcal{B}\left( \mathcal{K} \right)$, ${{A}_{1}},\ldots ,{{A}_{n}}$ be self-adjoint operators with spectra contained in $I$, and let $\left\{ {{w}_{i}} \right\},J,{{J}^{c}}$ be as in Theorem \ref{24}. Then
\begin{eqnarray*}
&& {{\mathcal{P}}_{f}}\left( \sum\limits_{i=1}^{n}{{{w}_{i}}{{\Phi }_{i}}\left( {{A}_{i}} \right)}\mid\sum\limits_{i=1}^{n}{{{w}_{i}}{{\Phi }_{i}}\left( {{B}_{i}} \right)} \right)  \le {{\omega }_{J}}{{\mathcal{P}}_{f}}\left( \frac{1}{{{\omega }_{J}}}\sum\limits_{i\in J}{{{w}_{i}}{{\Phi }_{i}}\left( {{A}_{i}} \right)}\mid\frac{1}{{{\omega }_{J}}}\sum\limits_{i\in J}{{{w}_{i}}{{\Phi }_{i}}\left( {{B}_{i}} \right)} \right)\\
&&+{{\omega }_{{{J}^{c}}}}{{\mathcal{P}}_{f}}\left( \frac{1}{{{\omega }_{{{J}^{c}}}}}\sum\limits_{i\in {{J}^{c}}}{{{w}_{i}}{{\Phi }_{i}}\left( {{A}_{i}} \right)}\mid\frac{1}{{{\omega }_{{{J}^{c}}}}}\sum\limits_{i\in {{J}^{c}}}{{{w}_{i}}{{\Phi }_{i}}\left( {{A}_{i}} \right)} \right)  \le \sum\limits_{i=1}^{n}{{{w}_{i}}{{\Phi }_{i}}\left( {{\mathcal{P}}_{f}}\left( {{A}_{i}}\mid{{B}_{i}} \right) \right)}.
\end{eqnarray*}
\end{corollary}

\section*{Acknowledgement}
The author (S.F.) was partially supported by JSPS KAKENHI Grant Number 16K05257.

\end{document}